\documentclass[12pt]{article}
\usepackage{amssymb,amsmath,amsthm,epsfig}
\usepackage{latexsym, enumerate}
\usepackage{eepic}
\usepackage{epic}
\usepackage{graphicx}
\usepackage{color}
\usepackage{ifpdf}
\usepackage{subfigure}
\usepackage{tikz}
\usepackage{dsfont}
\usepackage{multirow}
\usepackage{makecell}
\usepackage{algorithm}
\usepackage{bm}
\usepackage{multirow}
\usepackage{epstopdf}

\usepackage{cite}
\usepackage{hyperref}
\usepackage{xcolor}
\hypersetup{
  colorlinks,
  linkcolor={blue!80!black},
  urlcolor={blue!80!black},
  citecolor={blue!80!black}
}
\usepackage[capitalize,noabbrev]{cleveref}

\theoremstyle{plain}
\newtheorem{lem}{Lemma}[section]
\newtheorem{thm}[lem]{Theorem}

\theoremstyle{definition}
\newtheorem{defn}{Definition}[section]

\theoremstyle{remark}
\newtheorem{rem}{Remark}[section]

\newcommand{\zz}{\partial\Omega}
\begin{document}
\title{\MakeUppercase {\large\bf Determining a stationary mean field game system from full/partial boundary measurement} }

\author{
Ming-Hui Ding\thanks{Department of Mathematics, City University of Hong Kong, Kowloon, Hong Kong, China.\ \  Email: mingding@cityu.edu.hk}
\and
Hongyu Liu\thanks{Department of Mathematics, City University of Hong Kong, Kowloon, Hong Kong, China.\ \ Email: hongyu.liuip@gmail.com; hongyliu@cityu.edu.hk}
\and
Guang-Hui Zheng\thanks{School of Mathematics, Hunan University, Changsha 410082, China.\ \ Email: zhenggh2012@hnu.edu.cn; zhgh1980@163.com}
}

\date{}
\maketitle

\begin{center}{\bf ABSTRACT}
\end{center}\smallskip
In this paper, we propose and study the utilization of the Dirichlet-to-Neumann (DN) map to uniquely identify the discount functions $r, k$ and cost function $F$ in a stationary mean field game (MFG) system. This study features several technical novelties that make it highly intriguing and challenging.  Firstly, it involves a coupling of two nonlinear elliptic partial differential equations. Secondly, the simultaneous recovery of multiple parameters poses a significant implementation challenge. Thirdly, there is the probability measure constraint of the coupled equations to consider. Finally, the limited information available from partial boundary measurements adds another layer of complexity to the problem. Considering these challenges and problems, we present an enhanced higher-order linearization method to tackle the inverse problem related to the MFG system. Our proposed approach involves linearizing around a pair of zero solutions and fulfilling the probability measurement constraint by adjusting the positive input at the boundary. It is worth emphasizing that this technique is not only applicable for uniquely identifying multiple parameters using full-boundary measurements but also highly effective for utilizing partial-boundary measurements.

\smallskip
{\bf keywords}: Stationary mean field games; multiple parameters; unique identifiability; full/partial boundary measurement; CGO solutions.

\begin{center}\small
\tableofcontents
\end{center}

\section{Introduction}
\subsection{Mathematical setup and statement of main results}

We first introduce the mathematical setup of the study.
Let $\Omega$ be a smooth bounded domain of $\mathbb{R}^n, n\geq1$ with a  class $C^{2+\alpha}$ boundary for some $\alpha\in(0,1)$. We focus on the following stationary mean field games (MFG) system
\begin{align}
\label{um0}
\begin{cases}
\displaystyle-v\Delta u+\frac{1}{2}|\nabla u|^2+k(x)u=F(x,m)
\ \ \hspace*{1.4cm} &\mathrm{in}\  \Omega,\medskip\\
\displaystyle -v\Delta m-\mathrm{div}(m\nabla u)+r(x)m=0
\hspace*{2.50cm} &\mathrm{in}\ \Omega,\medskip\\
\displaystyle  u=f,\ m=g\hspace*{5.3cm} &\mathrm{on}\  \partial\Omega,
\end{cases}
\end{align}
where $\Delta$ and $\mathrm{div}$ are the Laplacian and divergence operator with respect to the $x$-variable, respectively. In this model, $u$ and $m$ represent the value function and the agent distribution function of the game, respectively. The coefficient $v$ is a positive constant, while the $k(x)$ and $r(x)$ are nonnegative functions. The cost function $F(x,m)$ signifies the interaction between the agent and the aggregate. On the boundary, the values of $u$ and $m$ are represented by the functions $f$ and $g$, respectively. It is worth highlighting that, as $m$ represents a probability distribution function, its non-negativity is essential. Further details about the model can be found in Section \ref{sub1.2}.

This paper aims to investigate the inverse problem of identifying the coefficients $k(x)$, $r(x)$, and the cost function $F(x,m)$ of the MFG \eqref{um0} by leveraging the knowledge of the Dirichlet-to-Neumann (DN) map on either the full or partial boundary. In this inverse problem, it is essential to guarantee that $m$ is non-negative. The non-negativity of the boundary value $g$ essentially implies the non-negativity of $m$, which can be easily deduced from the system \eqref{um0}. To address this inverse problem, we begin by establishing the DN mapping on the full boundary, which is defined as follows
\begin{align*}
\mathbf{{\Lambda}}_{k,r,F}:\ &\xi_1\times\xi_1^+\rightarrow C^{1+\alpha}(\partial\Omega)\times C^{1+\alpha}(\partial\Omega)\\
&(f,g)\ \mapsto(\partial_\nu u, \partial_\nu m){|}_{\partial \Omega},
\end{align*}
where $\xi_1=\{ f\in C^{2+\alpha}(\partial\Omega)|\|f\|_{C^{2+\alpha}(\partial\Omega)} <\delta_1\}$ and $\xi_1^+=\{ g\in C^{2+\alpha}(\partial\Omega)|g> 0 \ \mathrm{and} \ \|g\|_{C^{2+\alpha}(\partial\Omega)}$ $<\delta_1\}$ with a small $\delta_1$.  The notation that $\partial_\nu u(x)=\nabla u(x)\cdot\nu(x)$, where $\nu=(\nu_1, \nu_2,..., \nu_n)\in \mathbb{S}^{n-1}$ is a outer normal vector, and $\mathbb{S}^{n-1}$ denote a unit sphere of $\mathbb{R}^n$. Next, we introduce the partial DN map $\mathbf{\Lambda}^\mathbf{p}_{k,r,F}$ as follows:
\begin{align*}
\mathbf{\Lambda}^\mathbf{p}_{k,r,F}:\ &\xi_2\times\xi_2^+\rightarrow C^{1+\alpha}(\mathcal{U}_{-})\times C^{1+\alpha}(\mathcal{U}_{-})\\
&(f, g)\  \mapsto\ (\partial_\nu u,\partial_\nu m)|_{\mathcal{U}_{-}},
\end{align*}
where $\xi_2=\{ f\in C_0^{2+\alpha}(\mathcal{U}_{+})|\| f\|_{C^{2+\alpha}(\mathcal{U}_{+})} <\delta_2\}$ and $\xi_2^+=\{ g\in C_0^{2+\alpha}(\mathcal{U}_{+})| g> 0\ \ \mathrm{and}\ \| g\|_{C^{2+\alpha}(\mathcal{U}_{+})}$ $<\delta_2\}$ for a small $\delta_2$, here $\mathcal{U}_{\pm}$ denote a neighborhood of $\partial\Omega_{\pm}(\lambda')$, and the $\partial\Omega_{\pm}(\lambda')$ is defined by
\begin{align*}
\partial\Omega_{\pm}(\lambda')=\{ x\in\partial \Omega;\pm\lambda'\cdot\nu(x)>0\}
\end{align*}
with a given $\lambda'\in \mathbb{S}^{n-1}$. To that end, the inverse problems can be expressed in the following general form:
\begin{align}\label{IP1}
\mathbf{{\Lambda}}_{k,r,F}(f,g)\rightarrow k, r, F\ \   \mathrm{for}\  (f, g)\in\xi_1\times \xi_1^+,
\end{align}
and
\begin{align}\label{IP2}
\mathbf{\Lambda}^\mathbf{p}_{k,r,F}(f,g)\rightarrow k, r, F\ \   \mathrm{for}\  (f, g)\in\xi_2\times \xi_2^+.
\end{align}
In the inverse problems \eqref{IP1} and \eqref{IP2}, we concentrate on the theoretical unique identifiability issue. In the general formulation, we establish sufficient conditions to guarantee that \eqref{IP1} and \eqref{IP2} are uniquely identifiable, which can be formulated as follows
\begin{align*}
\small
\mathbf{{\Lambda}}_{k_1,r_1,F_1}(f,g)=\mathbf{{\Lambda}}_{k_2,r_2,F_2}(f,g)
\ \ \mathrm{if\ and\ only\ if}\  (k_1,r_1,F_1)=(k_2,r_2,F_2)\ \mathrm{for}\  (f, g)\in\xi_1\times \xi_1^+,
\end{align*}
and
\begin{align*}
\small
\mathbf{{\Lambda}}^\mathbf{p}_{k_1,r_1,F_1}(f,g)=\mathbf{{\Lambda}}^\mathbf{p}_{k_2,r_2,F_2}(f,g)
\ \ \mathrm{if\ and\ only\ if}\  (k_1,r_1,F_1)=(k_2,r_2,F_2)\ \mathrm{for}\  (f, g)\in\xi_2\times \xi_2^+,
\end{align*}
where the nonnegative functions $k_j, r_j\in C^{\alpha}(\overline{\Omega})$, and $F_j$ belongs in a priori admissible class $\mathcal{A}$ with $j=1,2$. For the priori admissible class $\mathcal{A}$, we also give its definition as follows
\begin{defn}\label{adm}
($\mathbf{Admissible\ class}$). We say $F(x,z):\Omega\times\mathbb{C}\rightarrow\mathbb{C}\in \mathcal{A}$, if it satisfies the following equations:\\
(i)\ The map $z\rightarrow F(\cdot,z)$ is holomorphic with value in $C^{\alpha}(\overline{\Omega})$;\\
(ii)\ $F(x,0)=\partial_zF(x,0)=0$ for all $x\in \overline{\Omega} $.
\end{defn}
From the Definition \ref{adm}, it means that $F(x,z)$ can be expanded into a power series as following:
\[
F(x,z)=\sum_{i=2}^{\infty}F^{(i)}(x)\frac{z^i}{i!},
\]
where $F^{(i)}(x)=\partial_z^iF(x,0)\in C^{\alpha}(\overline{\Omega})$.

Next, we can state the main result of the inverse problems \eqref{IP1} and \eqref{IP2}, which demonstrates that in a generic scenario, it is possible to recover the coefficients $k(x), r(x)$ and the running cost $F$ from the measurement map $\mathbf{{\Lambda}}_{k,r,F}$ or $\mathbf{{\Lambda}}^{\mathbf{p}}_{k,r,F}$.

\begin{thm}\label{alldata}
Let $\Omega\subset \mathbb{R}^n$ and $n\geq 3$. Assume $F_j$ in the admissible class $\mathcal{A}$, the nonnegative  functions $k_j, r_j\in C^{\alpha}(\overline{\Omega})$ for $j=1, 2$. Let $\mathbf{\Lambda}_{k_j,r_j,F_j}$ be the full DN map associated the  MFG \eqref{um0}.
If for any $f\in \xi_1, g\in \xi_1^+$, one has
\begin{align}\label{cod}
\mathbf{\Lambda}_{k_1,r_1,F_1}(f,g)=\mathbf{\Lambda}_{k_2,r_2,F_2}(f,g),
\end{align}
 then we have
 \[
 k_1=k_2, r_1=r_2 \ in\ \Omega;\ F_1=F_2\ in\ \Omega\times\mathbb{R^+}.
 \]
\end{thm}

\begin{thm}\label{pdata}
Let $\Omega\subset \mathbb{R}^n$ and $n\geq 3$. Assume $F_j$ in the admissible class $\mathcal{A}$, the nonnegative  functions $k_j, r_j\in C^{\alpha}(\overline{\Omega})$ for $j=1,2$, the sub-domain $\Omega'\subseteq \Omega$ satisfying $\partial\Omega\subseteq\partial \Omega'$ and $k_1=k_2, r_1=r_2,F_1=F_2$ in $\Omega'$. Let $\mathbf{\Lambda}^{\mathbf{p}}_{k_j,r_j,F_j}$ be the partial DN map associated the  MFG \eqref{um0}. If for any $f\in \xi_2, g\in \xi_2^+$, one has
\begin{align*}
\mathbf{\Lambda}^{\mathbf{p}}_{k_1,r_1,F_1}(f,g)=\mathbf{\Lambda}^{\mathbf{p}}_{k_2,r_2,F_2}(f,g),
\end{align*}
 then we have
 \[
 k_1=k_2, r_1=r_2 \ in\ \Omega;\ F_1=F_2\ in\ \Omega\times\mathbb{R^+}.
 \]
\end{thm}
\subsection{Background and motivation}\label{sub1.2}
MFG systems are a relatively new area of research in recent years, which aims to study the behavior of agents as their number tends to infinity. The theory was first introduced in the seminal publications of Lasry and Lions \cite {LL2004,LL2006,LL2007} and Huang-Caines-Malhame \cite {HMC2006,HMC2007,HCM2007} in 2006-2007, and has since received more attention and emerged from a growing amount of literature. The MFG system consists of the Hamilton-Jacobi-Bellman (HJB) equation coupled with the Kolmogorov-Fokker-Planck (KFP) equation, where the HJB equation represents the value function associated with the potential stochastic optimal control problem faced by the agent and the KFP equation is used to describe the probability density of agents in the game state space.
The MFG systems have strong applications in a wide range of areas such as economics, stock development, and transportation \cite{GLL2010,A2014,GPV2016,CD2018}. The forward problem for MFG systems has been extensively and intensively studied in the literature, and we refer to \cite{C2015,CDLL2019,ACD2020,CGPT2015} to obtain the well-posedness results on the forward problem associated with MFG systems in the current study.

 In this paper, we consider a general stationary MFG system \eqref{um0}, which emerges from the study of the behavior of infinite horizon MFG. This general form includes models in some special forms, such as smooth MFG models ($k$ and $r$ are vanishing)\cite{C2016,LL2004} and some discounted MFG models \cite{FGT2019, ABC2019}. In system \eqref{um0}, the Hamiltonian is the canonical quadratic form $|\nabla u|^2/2$, and we focus on the case that $F$ depends on $m$ locally. It is worth noting that in contrast to the periodic boundary conditions considered in \cite{LL2006}, we study MFG system with Dirichlet boundary conditions, which represent agents entering or exiting the game at a certain cost via $\partial\Omega$, and in fact, such MFG systems are more realistic settings \cite{OS2022,FGT2019}. In this model, $m$ is also not the usual probability density function, but it is emphasized that $m$ is non-negative. Agents can also leave the graveyard state at an exponential rate, which is encoded in the discount functions $k(x)$ and $r(x)$, where the discount appears in the $k(x)u$ term of the Hamilton-Jacobi equation and the $r(x)m$ term of the Kolmogorov-Fokker-Planck. The term involving $k$ and $r$
 in  \eqref{um0} are also relevant in the context of temporal semi-discretizations of time-dependent MFG systems.

 There is still very limited research on the inverse problem of MFG systems. There are only several numerical studies \cite{CFLO2022,DWO2022} and theoretical studies \cite{LMZ2023,LZ2022,LZ2023} for several MFG inverse problems via the boundary measurement associated with infinitely many events. There are also several recent works \cite{ILY,Kl23,KlAv,KLL1,KLL2,LY} on MFG inverse problems associated with a single-event measurement.  We note that at present, the study of non-negative constraints on the inverse problem associated with MFG systems is only rigorously discussed accordingly in \cite{LMZ2023, LZ2023, LZ2022} in several different contexts, as well as in the study of the coupled systems for biological populations\cite{LL2023,DHZ2023}. It is also pointed out that the previous studies on MFG systems have primarily investigated systems for a finite horizon with the Neumann or periodic boundaries. In these studies, the unknown parameters were uniquely identified using measurements of the full boundary, without considering the use of the more practically measured partial boundary data. Considering the importance of studying infinite horizon behavior and the practical need for measurements, we are mainly concerned with the inverse problem associated with MFG systems with the Dirichlet boundaries on infinite horizons. We utilize a more practical full or partial DN map to achieve the unique identifiability of multiple unknown coefficients $k(x), r(x), F(x,m)$ and still guarantee the non-negativity of $m$.

\subsection{Technical developments and discussion}\label{sub1.3}

As mentioned early, our focus is on studying the unique identifiability of the coefficients $k(x)$, $r(x)$, and the cost function $F(x,m)$ in the steady-state MFG system \eqref{um0}. Before investigating the inverse problem, we first examine the existence and uniqueness of solutions to the forward problem of this MFG system. The proofs regarding the existence and uniqueness of solutions for MFG systems rely on the implicit function theorem in complex spaces \cite{I1993,DT1998}. This approach currently provides only local well-posedness results for the forward problem, which are discussed in detail in Section \ref{sec2}.

For the study of the inverse problem, we mainly use an easily observable DN map on the full or partial boundary to simultaneously identify multiple parameters in the stationary MFG system. This imposes a significant challenge to the study of the inverse problem due to the non-linear coupling, simultaneous multi-parameter reconstruction, and non-negative constraints. Taking inspiration from the general use of higher-order linearization methods for solving nonlinear equations \cite{LLX2022,KU2020}, we adopt a similar approach to solve the inverse problem in a nonlinearly coupled system. Although the classical higher-order linearization method is a valuable tool, it fails to consider the non-negativity of $m$. This necessitates exploring a new form that can regulate the input of $g$ on the boundary and ensure the non-negativity of $m$.  Afterward, we require a linear expansion of the MFG system around a pair of zero solutions $(0,0)$. It is worth noting that linearizing at the $(0,0)$ solution simplifies the problem by breaking down the coupled nonlinear equations into several uncoupled  linear or non-linear equations. Furthermore, linearization around $(0,0)$ eliminates the need to construct new complex geometrical optics (CGO) solutions. In fact, knowing only the CGO solution of a single linear equation is sufficient to achieve the unique identifiability of the inverse problem. For linear expansions at nonzero solutions, please refer to \cite{LZ2022}, where new CGO solutions of the coupled equations must be constructed to achieve the unique identifiability of the inverse problem.

The rest of the paper is organized as follows. In Section \ref{sec2}, we present the well-posedness of the forward problem. Section \ref{sec3} is devoted to the study of the approximation and complex geometrical optics solutions of the linear elliptic equation. The proofs of the unique determination for the MFG systems are provided in Section \ref{sec4}.
\section{Well-posedness of the forward problem}\label{sec2}

The objective of this section is to demonstrate the well-posedness of the Dirichlet problem for a class of coupled elliptic equations with small boundary data. We begin by defining the H\"{o}lder space $C^{K+\alpha}(\overline{\Omega})$, which is defined as the subspace of $C^K(\overline{\Omega})$ such that $\phi\in C^{K+\alpha}(\overline{\Omega})$ if and only if $D^l\phi$ exists and are H\"{o}lder continuous with exponent $\alpha$ for all $l=(l_1,l_2,...,l_n)\in\mathbb{N}^n$ with $|l|\leq K$. Here, $D^l:=\partial_{x_1}^{l_1}\partial_{x_2}^{l_2}...\partial_{x_n}^{l_n}$ for $x=(x_1,x_2,...,x_n)$, and the norm is defined as follows:

\begin{align*}
\| \phi\|_{C^{K+\alpha}(\overline{\Omega})}:=\sum_{\mid l\mid=K}\sup_{x,y\in \Omega, x\neq y}\frac{|\partial^{l}\phi(x)-\partial^{l}\phi(y)|}{|x-y|^{\alpha}}
+\sum_{|l|\leq K}\|D^l \phi\|_{L^\infty(\Omega)}<\infty.
\end{align*}

Before presenting the proof of the local well-posedness of the MFG system \eqref{um0}, it is necessary to verify that 0 is not a Dirichlet eigenvalue for either $-v\Delta+r(x)$ or $-v\Delta+k(x)$.

\begin{thm}
Suppose that $F\in\mathcal{A}$, then there exist $\delta>0, C>0$ such that for any $f, g\in B_{\delta}(\partial \Omega):=\{ \check{f}\in C^{2+\alpha}(\partial\Omega): \|\breve{f}\|_{C^{2+\alpha}(\partial\Omega)}<\delta\}$, the system \eqref{um0} has a solution which satisfies
\begin{align*}
\|(u,m)\|_{C^{2+\alpha}(\overline{\Omega})\times C^{2+\alpha}(\overline{\Omega})}\leq C (\| f\|_{C^{2+\alpha}(\partial\Omega)}+\| g\|_{C^{2+\alpha}(\partial\Omega)}).
\end{align*}\\
 Furthermore, the solution $(u,m)$ is unique within the class
\begin{align*}
\{(u,m)\in{C^{2+\alpha}(\overline{\Omega})\times C^{2+\alpha}(\overline{\Omega})}:\ \| (u,m)\|_{C^{2+\alpha}(\overline{\Omega})\times C^{2+\alpha}(\overline{\Omega})}\leq C\delta\}.
\end{align*}

\end{thm}
\begin{proof}
Let
\begin{align*}
Y_1:&=\{ (f,g)\in C^{2+\alpha}(\partial\Omega)\times C^{2+\alpha}(\partial\Omega) \},\\
Y_2:&=\{ (u,m)\in C^{2+\alpha}(\overline{\Omega})\times C^{2+\alpha}(\overline{\Omega})\},\\
Y_3:&=Y_1\times C^{\alpha}(\overline{\Omega})\times C^{\alpha}(\overline{\Omega})
\end{align*}
and we consider the following mapping
$$\mathfrak{F}:Y_1\times Y_2\rightarrow Y_3  $$ for any $(\tilde{f},\tilde{g},\tilde{u},\tilde{m})\in Y_1\times Y_2$, we have
\begin{align*}
\mathfrak{F}(\tilde{f},\tilde{g},\tilde{u},\tilde{m})(x):=&\big{(}\tilde{u}|_{\partial\Omega}-\tilde{f}, \tilde{m}|_{\partial\Omega}-\tilde{g}, -v\Delta \tilde{u}+\frac{1}{2}|\nabla \tilde{u}|^2+k(x)\tilde{u}-F(x,\tilde{m}),\\ &-v\Delta \tilde{m}-\mathrm{div}(\tilde{m}\nabla \tilde{u})+r(x)\tilde{m}\big{)}.
\end{align*}
We claim that $\mathfrak{F}$ is well-defined. Since the H\"{o}lder space is an algebra under point-wise multiplication, we find $|\nabla \tilde{u}|^2\in C^{\alpha}(\overline{\Omega})$ and $\mathrm{div}(\tilde{m}\nabla \tilde{u})\in C^{\alpha}(\overline{\Omega})$. We only have to prove that $F(x,\tilde{m})\in C^{\alpha}(\overline{\Omega})$.  By the Cauchy's estimates, we deduce that
\begin{align}\label{Fk}
\|F^{(i)}\|_{C^\alpha(\overline{\Omega})}\leq\frac{i!}{R^i}\sup_{|z|=R}\|F(\cdot,z)\|_{C^{\alpha}(\overline{\Omega})}\ \ \mathrm{for}\  R>0.
\end{align}
Thus, for all $i\in\mathbb{{N}}$, there exists a constant $L$ such that
\begin{align}\label{Fk|}
\bigg{\|}\frac{F^{(i)}}{i!}z^i\bigg{\|}_{C^{\alpha}(\overline{\Omega})}\leq\frac{L^i}{R^i}\| z\|^i_{C^\alpha(\overline{\Omega})}\sup_{|z|=R}\|F(\cdot,z)\|_{C^{\alpha}(\overline{\Omega})}.
\end{align}
For choosing $R=2L\|z\|_{C^{\alpha}(\overline{\Omega})}$, we find the series $\sum_{k=2}^\infty F^{(i)}\frac{z^i}{i!}$ converges in $C^{\alpha}(\overline{\Omega})$ and therefore $F(x,\tilde{m})\in C^\alpha(\overline{\Omega})$. This indicates that the claim is valid.

Next, we prove that $\mathfrak{F}$ is holomorphic. Since $\mathfrak{F}$ is locally bounded, it suffices to prove that it is weak holomorphy. Let $(f_0,g_0,u_0, m_0)$ , $(f_1,g_1,u_1, m_1)\in Y_1\times Y_2$, we show that the following mapping
\begin{align*}
\beta\in\mathbb{C}\mapsto \mathfrak{F}((f_0,g_0,u_0, m_0)+\beta(f_1,g_1,u_1, m_1))
\end{align*}
is holomorphic in $\mathbb{C}$ with values in $Y_3$. It is evident that we only have to check that the map $\beta\rightarrow F(x, m_0+\beta m_1)$ is holomorphic in $\mathbb{C}$ with values in $C^{\alpha}(\overline{\Omega})$. We derive this fact due to the convergence of the series
\begin{align*}
\sum_{i=2}^\infty\frac{F^{(i)}}{i!}( m_0+\beta m_1)^i
\end{align*}
in $C^{\alpha}(\overline{\Omega})$, locally uniformly in $\beta\in\mathbb{C}$.

Note that $\mathfrak{F}(0,0,0,0)=0$ and the differential $\partial_{\tilde{u},\tilde{m}}\mathfrak{F}(0,0,0,0)$ is given by
$$\partial_{\tilde{u},\tilde{m}}\mathfrak{F}(0,0,0,0)(u,m)=(u|_{\partial \Omega}, m|_{\partial \Omega}, -v\Delta u+k(x)u , -v\Delta m+r(x)m).$$
As 0 is not a Dirichlet eigenvalue of $-v\Delta+r(x)$ nor of $-v\Delta+k(x)$, we can easily deduce that
 $\partial_{\tilde{u},\tilde{m}}\mathfrak{{F}}(0,0,0,0)$ is a linear isomorphism between $Y_2$ and $Y_3$. According to the implicit function theorem, we obtain that there exists a unique holomorphic function $S:B_\delta(\partial\Omega)\times B_\delta(\partial\Omega)\rightarrow Y_2$ such that $\mathfrak{{F}}(f,g,S(f,g))=0$ for all $f,g\in B_\delta(\partial\Omega)$.  Let $(u,m)=S(f,g)$, we obtain the unique solution of the MFG \eqref{um0}. Setting $(0,0)=S(f_0,g_0)$, since $S$ is Lipschitz, we drive that there exist constants $C>0$ such that
\begin{align*}
\parallel(u,m)\parallel_{C^{2+\alpha}(\overline{\Omega})\times C^{2+\alpha}(\overline{\Omega}) }
\leq C(\parallel f\parallel_{C^{2+\alpha}(\partial\Omega)}+\parallel g\parallel_{C^{2+\alpha}(\partial\Omega)}).
\end{align*}
This proof is complete.\\
\end{proof}

\section{Runge approximation and complex geometrical optics solutions}\label{sec3}

In this section, we establish the Runge approximation property for the linear elliptic equation and the complex geometrical optics (CGO) solutions with a focus on special boundary conditions. These properties, together with the construction of CGO solutions, will be used to prove Theorems \ref{alldata} and \ref{pdata}.
\subsection{Runge approximation}

\begin{lem}\label{fdata}
Suppose that $k(x)\in C^{\alpha}(\overline{\Omega})$. Then for any solution $V\in H^1(\Omega)$ to
$-v\Delta V+k(x)V=0\ \mathrm{in}\ \Omega$, there exist solution $U\in C^{2+\alpha}(\overline{\Omega})$ to
$ -v\Delta U+k(x)U=0\ \mathrm{in}\ \Omega,$
such that  for $\eta>0$
\begin{align*}
\parallel V-U\parallel_{L^2(\Omega)}<\eta.
\end{align*}
\end{lem}
\begin{proof}
Define
 $$X=\{ U\in C^{2+\alpha}(\overline{\Omega})|\ U\ is\ a\ solution\ to\ -v\Delta U+k(x)U=0\ \mathrm{in}\ \Omega\},$$
and $$Y=\{ V\in H^{1}({\Omega})|\ V\ is\ a \ solution\ to\ -v\Delta V+k(x)V=0\ \mathrm{in}\ \Omega\}.$$
Our goal is to prove the density of $X$ in $Y$. Utilizing the Hahn-Banach theorem, we can simplify this task to verifying the following statement: if $y\in L^2(\Omega)$ satisfies
$$\int_\Omega yUdx=0 \ \ for\ any\ U\in X,$$
then
$$\int_\Omega yVdx=0 \ \ for\ any\ V\in Y.$$
Let $y \in L^2(\Omega)$ and  and satisfies $\int_\Omega yUdx=0$ for any $U\in X$. Then, we consider the following equation
\begin{align*}
\begin{cases}
\displaystyle -v\Delta \overline{U}+k(x)\overline{U}=y\ &\mathrm{in}\ \Omega,\\
\displaystyle \overline{U}=0\ &\mathrm{on}\ \partial\Omega,
\end{cases}
\end{align*}
and the solution  $\overline{U} \in H^2(\Omega)$. For any $U\in X$, we find
$$0=\int_{\Omega}yUdx=\int_{\Omega}(-v\Delta \overline{U}+k(x)\overline{U})Udx=-\int_{\partial\Omega}
v\partial_\nu \overline{U}UdS.$$
Given that $U|_{\partial\Omega}$ can be any arbitrary function with compact support on $\partial\Omega$, it is necessary for $\partial_\nu \overline{U}$ to be equal to zero. Thus, for any $V\in Y$,
$$\int_{\Omega}yVdx=\int_{\Omega}(-v\Delta \overline{U}+k(x)\overline{U})Vdx=-\int_{\partial\Omega}
v\partial_\nu \overline{U}VdS=0,$$
which verifies the assertion.\\
\end{proof}

\begin{lem}\label{prung}
Let $k\in C^{\alpha}(\overline{\Omega})$. Then for any solution $W\in H^1(\Omega)$ to the equation $-\nu\Delta W+k(x)W=0$, there exist a solution $U\in C^{2+\alpha}(\overline{\Omega})$ to
\begin{align}\label{p2U}
\begin{cases}
\displaystyle\  -v\Delta {U}+k(x){U}=0\ &\mathrm{in}\ \Omega,\\
\displaystyle\ {U}=0\ &\mathrm{on}\ \Gamma
\end{cases}
\end{align}
with  $\Gamma$ is a open subset of $\partial\Omega$, for any $\eta>0$  such that
$$\parallel W-U\parallel_{L^2(\tilde{\Omega})}<\eta,$$
where $\tilde{\Omega}=\Omega\setminus\Omega'$ and the sub-domain $\Omega'\subseteq \Omega$ satisfying $\partial\Omega\subseteq\partial \Omega'$.
\begin{proof}
Let $$ X'=\{ U\in C^{2+\alpha}(\overline{\Omega})|\ U\ is\ a\ soliution\ to \ \eqref{p2U}\},$$
and
$$Y'=\{ W\in  H^1({\Omega})| W\ is\ a\ solution\ to -\nu\Delta W+k(x)W=0\}.$$
Our objective is to demonstrate the density of $X'$ in $Y'$. Applying the Hahn-Banach theorem again, it is enough to establish that if $y\in L^2(\tilde{\Omega})$ satisfies
$$\int_{\tilde{\Omega}}yUdx=0,\ for\ any\ U\in X',$$
then
$$\int_{\tilde{\Omega}}yWdx=0,\ for\ any\ W\in Y'.$$
We extend $y$ to $\Omega$ by letting $y=0$ outside $\tilde{\Omega}$.
Consider
\begin{align}\label{U2}
\begin{cases}
\displaystyle\  -v\Delta {\overline{U}}+k(x){\overline{U}}=y\ \ \ &\mathrm{in}\ \Omega,\\
\displaystyle\ {\overline{U}}=0\ &\mathrm{on}\ \partial\Omega,
\end{cases}
\end{align}
thus for any $U\in X'$, we have
$$0=\int_{\Omega}yUdx=\int_{\Omega}(-v\Delta {\overline{U}}+k(x){\overline{U}})Udx=-\int_{\partial\Omega}v\partial_\nu\overline{U}UdS
=-\int_{\partial\Omega\setminus \Gamma}v\partial_\nu\overline{U}UdS.$$
Since $U|_{\partial\Omega\setminus \Gamma}$ can be arbitrary function, then $\partial_\nu \overline{U}=0$ on $\partial\Omega\setminus \Gamma$.  Notice that
$$-v\Delta\overline{U}+k(x)\overline{U}=0\ \ \mathrm{in}\ \Omega'.$$
Due to $\Omega'$ is open and connected, according to the unique continuation principle for linear elliptic equations, we can conclude that $\overline{U}=0$ on $\Omega'$.
Hence, $\overline{U}|_{\partial\Omega'}=\partial_\nu \overline{U}|_{\partial\Omega'}=0$, and it follows that
$$\overline{U}|_{\partial\tilde{\Omega}}=\partial_\nu \overline{U}|_{\partial\tilde{\Omega}}=0.$$
Thus, for any $W\in Y'$, we have
$$\int_{\tilde{\Omega}}yWdx=\int_{\tilde{\Omega}}(-v\Delta {\overline{U}}+k(x){\overline{U}})Wdx=-\int_{\partial\tilde{\Omega}}
v\partial_{\nu}\overline{U}WdS=0.$$
This proof is complete.\\
\end{proof}
\end{lem}

\begin{rem}
Lemma \ref{prung} defines the domain ${\Omega'}$ as ${\Omega'}=\Omega\setminus \tilde{\Omega}$, where $\tilde{\Omega}$ is a subset of $\Omega$, and $\Omega'$ is a connected region. Moreover, we can select $\tilde{\Omega}$ to be as large as possible, resulting in $\Omega'$ being a narrow domain. This means that for subsequent uniqueness identification of unknown coefficients using partial boundary data, it is only necessary to have prior knowledge of the coefficients near the partial boundary.

\end{rem}

\subsection{Construction of CGO solutions}

We first introduce the subset of the boundary
\begin{align*}
\partial\Omega_{+,\varepsilon_0}(\lambda)&=\{ x\in\partial \Omega;\lambda\cdot\nu(x)>\varepsilon_0\},\\
\partial\Omega_{-,\varepsilon_0}(\lambda)&=\{ x\in\partial \Omega;\lambda\cdot\nu(x)<\varepsilon_0\},
\end{align*}
where $\lambda\in \mathbb{S}^{n-1}$ and $\varepsilon_0$ is a small constant. Setting $\varphi=\lambda\cdot x$ for $x\in\Omega$, we consider the weighted Hilbert space $L^2(\Omega; e^{\varphi/h})$ and $L^2(\Omega;e^{-\varphi/h})$ with the scalar products
\begin{align*}
\langle \tilde{u},\hat{u} \rangle_{e^{\varphi/h}}&=\int_{\Omega}\tilde{u}(x) \overline{\hat{u}(x)}e^{2\varphi/h}dx,\\
\langle \tilde{u},\hat{u} \rangle_{e^{-\varphi/h}}&=\int_{\Omega}\tilde{u}(x) \overline{\hat{u}(x)}e^{-2\varphi/h}dx,
\end{align*}
respectively. We now introduce the following auxiliary Lemma \cite{BU2002,S2008}, which plays a crucial role in the construction of CGO solutions.
\begin{lem}\label{Car-est}
(Carleman estimate with boundary terms) Let $k\in C^\alpha(\overline{\Omega})$, and $\varphi(x)=\lambda\cdot x$ with $\lambda\in \mathbb{S}^{n-1}$. There exist constants $C>0$ and $h_0>0$ such that whenever $0<h\leq h_0$, we have
\begin{equation*}
\begin{aligned}
&h\int_{\partial\Omega_{+}(\lambda)}(\lambda\cdot\nu ) e^{-2\varphi/h}|\partial_\nu V|^2 dS+\int_{\Omega}e^{-2\varphi/h}| V|^2dx\\
\leq& Ch^2\int_{\Omega}e^{-2\varphi/h}|(-\Delta+k)V|^2dx
-Ch\int_{\partial\Omega_{-}(\lambda)}(\lambda\cdot\nu ) e^{-2\varphi/h}|\partial_\nu V|^2 dS
\end{aligned}
\end{equation*}
for any $V\in C^2(\overline{\Omega})$ with $V|_{\partial\Omega}=0$.
\end{lem}

\begin{lem}\label{part-sol}Let $k(x)\in C^\alpha(\overline{\Omega})$, $\lambda, \eta\in \mathbb{S}^{n-1}$ and $\lambda\cdot\eta=0$. Assume that $\varphi=\lambda\cdot x, \psi=\eta\cdot x $ with $x\in\Omega$.

(i)\ For any function $a\in
H^2(\Omega)$ satisfying
$$(\lambda+\mathrm{i}\eta)\cdot\nabla a=0\ \ in\ \Omega,$$
 then the equation $-v\Delta u+k(x)u=0$ in $\Omega$ has a solution
 \begin{align*}
 u(x)=e^{\frac{1}{h}(\mathrm{i}\psi+\varphi)}(a+b_1),
 \end{align*}
where $b_1\in H^1(\Omega)$ satisfies
$$\lim_{h\rightarrow 0}\|b_1\|_{L^2(\Omega)}=0,\ \ 0<h<h_0$$
for some $h_0>0$.

(ii)\ We can find solutions to
 \begin{align*}
\begin{cases}
\displaystyle-v\Delta u+k(x)u=0
\ \ \hspace*{1.0cm} &\mathrm{in}\  \Omega,\medskip\\
\displaystyle  u=0,\hspace*{4.3cm} &\mathrm{on}\  \partial\Omega_{+,\varepsilon_0}(\lambda)
\end{cases}
\end{align*}
 of the form
\begin{align*}
u(x)=e^{-\frac{1}{h}(\mathrm{i}\psi+\varphi)}(1+b_2)
\end{align*}
and satisfies
 $$\lim_{h\rightarrow 0}\|b_2\|_{L^2(\Omega)}=0,\ \  0<h<h_0$$
  for some $h_0>0.$
\end{lem}
\begin{proof}
(i)The proof is given in \cite{CFKKU2021}.

(ii) We first define
$$\mathcal{D}=\big{\{}V\in C^2(\overline{\Omega})\big{|} V|_{\zz}=0\ \mathrm{and}\ \frac{\partial V}{\partial\nu}\bigg{|}_{\partial\Omega_{-,\varepsilon_0}(\lambda)}=0\big{\}}.$$
For any $Q\in L^2(\Omega; e^{\varphi/h})$ and $q\in L^2(\zz; e^{\varphi/h})$, and
$k\in C^\alpha(\overline{\Omega})$, we define a linear functional $\mathcal{L}$ on the linear subspace $\mathcal{L}=(-v\Delta+k)\mathcal{D}$ as follows:
$$\mathcal{L}((-v\Delta+k)V)=\langle V, Q\rangle-\langle v\partial_\nu V, q\rangle_{\partial\Omega_{+,\varepsilon_0}(\lambda)}.$$
From the Lemma \ref{Car-est}, we have
\begin{equation*}
 \begin{aligned}
 |\mathcal{L}((-v\Delta+k)V)|&\leq|\langle V, Q\rangle|+|\langle v\partial_\nu V, q\rangle_{\partial\Omega_{+,\varepsilon_0}(\lambda)}|\\
 &\leq\|V\|_{L^2(\Omega; e^{-\varphi/h})}\|Q\|_{L^2(\Omega; e^{\varphi/h})}+\|v\partial_\nu V\|_{L^2(\partial\Omega_{+,\varepsilon_0}(\lambda);e^{-\varphi/h})}
 \|q\|_{L^2(\partial\Omega_{+,\varepsilon_0}(\lambda);e^{\varphi/h})}\\
 &\leq C(h^2\|Q\|_{L^2(\Omega; e^{\varphi/h})}+h\|q\|_{L^2(\partial\Omega_{+,\varepsilon_0}(\lambda);e^{\varphi/h}}))
 \|(-\Delta+k)V\|_{L^2(\Omega; e^{-\varphi/h})}^2,
 \end{aligned}
\end{equation*}
we conclude that $\mathcal{L}$ is bounded, and by Hahn-Banach theorem, it has an extension $\mathcal{L}$ to a linear functional in $L^2(\Omega; e^{-\varphi/h})$ with the same norm. Therefore there exists $\rho\in L^2(\Omega; e^{\varphi/h})$ such that
 \begin{align}\label{fv}
\langle(-v\Delta+k)V,\rho\rangle=\langle V, Q\rangle+\langle v\partial_\nu V, q\rangle_{\partial\Omega_{+,\varepsilon_0}(\lambda)}
 \end{align}
and
$$\|\rho\|_{L^2(\Omega; e^{\varphi/h})}\leq C(h^2\|Q\|_{L^2(\Omega; e^{\varphi/h})}+h\|q\|_{L^2(\partial\Omega_{+,\varepsilon_0}(\lambda);e^{\varphi/h})}).$$
From \eqref{fv}, we derive that $\rho$ is an $L^2(\Omega)$ solution to the following equation
 \begin{align}
\label{ume1}
\begin{cases}
\displaystyle-v\Delta \rho+k(x)\rho=Q
\ \ \hspace*{1.0cm} &\mathrm{in}\  \Omega,\medskip\\
\displaystyle  \rho=q\hspace*{4.3cm} &\mathrm{on}\  \partial\Omega_{+,\varepsilon_0}(\lambda)
\end{cases}
\end{align}
with $Q\in L^2(\Omega; e^{\varphi/h})$ and $q\in L^2(\partial\Omega; e^{\varphi/h})$. In the equation \eqref{ume1}, taking $Q=-k(x)e^{-\frac{1}{h}(\mathrm{i}\psi+\varphi)}\in L^2(\Omega; e^{\varphi/h})$ and $q=-e^{\frac{1}{h}(\mathrm{i}\psi+\varphi)}\in L^2(\zz; e^{\varphi/h})$, we can easy to see that $$\rho=u(x)-e^{-\frac{1}{h}(\mathrm{i}\psi+\varphi)}=e^{-\frac{1}{h}
(\mathrm{i}\psi+\varphi)} b_2 \in L^2(\Omega; e^{\varphi/h})$$ is a solution of the equation \eqref{ume1}, and
$$\lim_{h\rightarrow 0}\|b_2\|_{L^2(\Omega)}=\lim_{h\rightarrow 0}\|\rho\|_{L^2(\Omega;e^{\varphi/h})}=0.$$
The proof is complete.
\end{proof}

\section{Simultaneous recovery of multiple parameters}\label{sec4}
\subsection{Higher-order linearization methods}
We introduce the fundamental framework of the higher-order linearization method, which is based on the infinite differentiability of the solution concerning a set of input functions $f$ and $g$. Setting
\begin{align}\label{fg}
f(x)=\sum_{l=1}^N\epsilon_lf_l\ \  \mathrm{and} \ g(x)=\sum_{l=1}^N\epsilon_lg_l \ \ \mathrm{for}\ x\in\partial\Omega\ 
\end{align}
where $f_l\in\xi_1$ for $l=1,2,...,N$, $g_l\in \xi_1^+$ for $l\neq 2$, and $\epsilon=(\epsilon_1,\epsilon_2,..., \epsilon_N )\in (\mathbb{R}^{N})^+$ with $|\epsilon|=\sum_{l=1}^N|\epsilon_l|$ small enough.   It is important to note that $g_2$ can be chosen to be an arbitrary function in $C^{2+\alpha}(\partial\Omega)$, and $\epsilon_2g_2$ can be made sufficiently small such that $g(x)>0$. This input is crucial in overcoming the challenges associated with the inverse problem, refer to the proofs in the subsection \ref{proof1.1} and \ref{proof1.2}.
By the well-posedness of the forward problem in the Section \ref{sec2}, then there exist unique solution $(u,m)\in C^{2+\alpha}(\overline{\Omega})\times C^{2+\alpha}(\overline{\Omega})$ to the following system
\begin{align}
\label{ume}
\begin{cases}
\displaystyle-v\Delta u+\frac{1}{2}|\nabla u|^2+k(x)u=F(x,m)
\ \ \hspace*{1.0cm} &\mathrm{in}\  \Omega,\medskip\\
\displaystyle -v\Delta m-\mathrm{div}(m\nabla u)+r(x)m=0
\hspace*{1.0cm} &\mathrm{in}\ \Omega,\medskip\\
\displaystyle  u=\sum_{l=1}^N\epsilon_lf_l,\ m=\sum_{l=1}^N\epsilon_lg_l \hspace*{4.3cm} &\mathrm{on}\  \partial\Omega.
\end{cases}
\end{align}
Let $(u(x;\epsilon), m(x;\epsilon))$ be the solution of \eqref{ume}, when $\epsilon=0$, we have
\[
(u(x;0),m(x;0))=(0,0)\ \ \mathrm{for}\ x\in \Omega.
\]
Let
\begin{align*}
u^{(1)}:&=\partial_{\epsilon_1}u|_{\epsilon=0}=\lim_{\epsilon\rightarrow 0} \frac{u(x;\epsilon)-u(x;0)}{\epsilon_1},\\
m^{(1)}:&=\partial_{\epsilon_1}m|_{\epsilon=0}=\lim_{\epsilon\rightarrow 0} \frac{m(x;\epsilon)-m(x;0)}{\epsilon_1},
\end{align*}
we obtain that $(u^{(1)}, m^{(1)})$ satisfies the following system:
\begin{align*}
\begin{cases}
\displaystyle-v\Delta u^{(1)}+k(x)u^{(1)}=0
\ \ \hspace*{1.4cm} &\mathrm{in}\  \Omega,\medskip\\
\displaystyle -v\Delta m^{(1)}+r(x)m^{(1)}=0
\hspace*{1.50cm} &\mathrm{in}\ \Omega,\medskip\\
\displaystyle  u^{(1)}=f_1, m^{(1)}=g_1>0\hspace*{5.3cm} &\mathrm{on}\  \partial\Omega.
\end{cases}
\end{align*}
Similarly, we can get
\begin{align*}
u^{(2)}:&=\partial_{\epsilon_2}u|_{\epsilon=0}=\lim_{\epsilon\rightarrow 0} \frac{u(x;\epsilon)-u(x;0)}{\epsilon_2},\\
m^{(2)}:&=\partial_{\epsilon_2}m|_{\epsilon=0}=\lim_{\epsilon\rightarrow 0} \frac{m(x;\epsilon)-m(x;0)}{\epsilon_2},
\end{align*}
satisfies the following system:
\begin{align*}
\begin{cases}
\displaystyle-v\Delta u^{(2)}+k(x)u^{(2)}=0
\ \ \hspace*{1.4cm} &\mathrm{in}\  \Omega,\medskip\\
\displaystyle -v\Delta m^{(2)}+r(x)m^{(2)}=0
\hspace*{1.50cm} &\mathrm{in}\ \Omega,\medskip\\
\displaystyle  u^{(2)}=f_2, m^{(2)}=g_2\hspace*{5.3cm} &\mathrm{on}\  \partial\Omega.
\end{cases}
\end{align*}
Note that $g_2$ is arbitrary. Then, we have the second-order linearization as follows:
\[
u^{(1,2)}:=\partial_{\epsilon_1}\partial_{\epsilon_2}u|_{\epsilon=0},
m^{(1,2)}:=\partial_{\epsilon_1}\partial_{\epsilon_2}m|_{\epsilon=0},
\]
thus we find
\begin{align*}
\begin{cases}
\displaystyle-v\Delta u^{(1,2)}+\nabla u^{(1)}\cdot \nabla u^{(2)}+k(x)u^{(1,2)}=F^{(2)}m^{(1)}m^{(2)}
\ \ \hspace*{1.4cm} &\mathrm{in}\  \Omega,\medskip\\
\displaystyle -v\Delta m^{(1,2)}-\mathrm{div}(m^{(1)}\nabla u^{(2)})-\mathrm{div}(m^{(2)}\nabla u^{(1)})+r(x)m^{(1,2)}=0
\hspace*{1.50cm} &\mathrm{in}\ \Omega,\medskip\\
\displaystyle  u^{(1,2)}=m^{(1,2)}=0\hspace*{5.3cm} &\mathrm{on}\  \partial\Omega.
\end{cases}
\end{align*}
It is worth noting that the non-linear terms of the second-order linearized system are dependent solely on the first-order linearized system. As a result, to determine the higher-order Taylor coefficients, we consider
\begin{align*}
u^{(1,2,...,N)}:&=\partial_{\epsilon_1}\partial_{\epsilon_2}...\partial_{\epsilon_N}u|_{\epsilon=0},\\
m^{(1,2,...,N)}:&=\partial_{\epsilon_1}\partial_{\epsilon_2}...\partial_{\epsilon_N}m|_{\epsilon=0}.
\end{align*}
In a similar manner, we can obtain a series of elliptic systems that can be used to recover the higher-order Taylor coefficients $F^{(i)}$ for $i\in{N}$.

\subsection{Proof of Theorem \ref{alldata}}\label{proof1.1}

$\mathbf{Step\ I}$. We begin by demonstrating the uniqueness of $k(x)$ and $r(x)$. Since the proof of the uniqueness of $k(x)$ is analogous to that of $r(x)$, we only provide the approach for proving the uniqueness of $r(x)$. Using first-order linearization for \eqref{ume} around the solution $(u_0, m_0)=(0,0)$. For $j=1,2$,  $u_j^{(2)}$ and $m_j^{(2)}$ satisfy the following equations
\begin{align}
\label{um1j}
\begin{cases}
\displaystyle-v\Delta u_j^{(2)}+k_ju_j^{(2)}=0
\ \ \hspace*{1.0cm} &\mathrm{in}\  \Omega,\medskip\\
\displaystyle -v\Delta m_j^{(2)}+r_jm_j^{(2)}=0
\hspace*{1.0cm} &\mathrm{in}\ \Omega,\medskip\\
\displaystyle  u_j^{(2)}=f_2, \  m_j^{(2)}=g_2\hspace*{3.3cm} &\mathrm{on}\  \partial\Omega.
\end{cases}
\end{align}
Note that $g_2\in C^{2+\alpha}(\partial\Omega)$ is arbitrary. Subtracting the second equation of \eqref{um1j} with $j=1,2$, we have
\begin{align}
\label{wk1}
\begin{cases}
\displaystyle-v\Delta w+r_1w=(r_2-r_1)m_2^{(2)}
\ \ \hspace*{1.4cm} &\mathrm{in}\  \Omega,\medskip\\
\displaystyle  w=0\hspace*{5.3cm} &\mathrm{on}\  \partial\Omega,
\end{cases}
\end{align}
where $w=m_1^{(2)}-m_2^{(2)}$. Let $z$ is a solution to $-v\Delta z+r_1z=0$ in $\Omega$.  Multiplyig both sides of \eqref{wk1} by $z$ and integrating over $\Omega$, according to  Green's formula and  the fact $\partial_\nu m_1^{(2)}=\partial_\nu m_2^{(2)}$ from \eqref{cod}, we have
 \begin{align*}
 &\int_{\Omega}(r_2-r_1)m_2^{(2)}zdx=0.
\end{align*}
From the Lemma \ref{part-sol}, it follows that there exist solutions $m_2^{(2)}\in H^1(\Omega), z\in H^1(\Omega)$  of the following form
\begin{equation}
\label{CGO}
\begin{aligned}
m_2^{(2)}:&=e^{-\frac{1}{h}(\mathrm{i}\psi+\varphi)}(e^{\mathrm{i}x\cdot\xi}+b_1),\\
z:&=e^{\frac{1}{h}(\mathrm{i}\psi+\varphi)}(1+b_2),
\end{aligned}
\end{equation}
where $\|b_j\|_{L^2(\Omega)}\leq Ch$ for $j=1,2$, $\psi=\eta\cdot x$ and $\varphi=\lambda\cdot x$ with $\lambda, \eta\in \mathbb{S}^{n-1}$ and $\lambda\perp \eta $. According to Lemma \ref{part-sol}, we choose $a=e^{{\mathrm{i}x\cdot\xi}}$ for the first solution of \eqref{CGO}, with $\xi\in \mathbb{R}^n$ and $\{ \lambda,\eta,\xi\}$ forming an orthogonal triplet (we need the dimension $n\geq 3$). For the second solution, we choose $a=1$.

Then, by the Lemma \ref{fdata}, there exist the sequences of solution $\hat{m}_l\in C^{2+\alpha}(\overline{\Omega})$ of the second equation in \eqref{um1j} and $\hat{z}_l\in C^{2+\alpha}(\overline{\Omega})$ to the equation $-v\Delta z+r_1z=0$, such that
\begin{align*}
\lim_{l\rightarrow\infty}\|\hat{m}_l-m_2^{(2)}\|_{L^2(\Omega)}=0,\
\lim_{l\rightarrow\infty}\|\hat{z}_l-z\|_{L^2(\Omega)}=0,
\end{align*}
thus we have
\begin{align*}
 \int_{\Omega}(r_2-r_1)\hat{m}_l\hat{z}_ldx=0.
\end{align*}
Let $l\rightarrow\infty$, and combing with \eqref{CGO}, we obtain
 \begin{align}\label{k12r12}
 \int_{\Omega}(r_2-r_1)(e^{\mathrm{i}\xi\cdot x}+b_1)(1+b_2)dx=0,
\end{align}
as $h\rightarrow0$ in \eqref{k12r12}, we have
 \begin{align}\label{fullr}
 \int_{\Omega}(r_2-r_1)e^{\mathrm{i}\xi\cdot x}dx=0 \ \  \mathrm{for}\ \xi\in\mathbb{R}^n.
\end{align}
 Setting
 $\tilde{r}=r_1-r_2$ in $\Omega$ and vanish outside $\Omega$, from \eqref{fullr}, we have the Fourier transform of $\tilde{r}$ is zero for every $\xi \in \mathbb{R}^n$, we conclude that $\tilde{r}=0$. This implies $r_2=r_1$ in $\Omega$.

\medskip

$\mathbf{Step\ II}$. Fix $r_1=r_2, k_1=k_2$. We consider a second-order linearization. For $j=1,2$, we have $u_j^{(1,2)}$ and  $m_j^{(1,2)}$ satisfy the following equations
\begin{align}
\label{fum}
\begin{cases}
\displaystyle-v\Delta u_j^{(1,2)}+\nabla u^{(1)}\cdot \nabla u^{(2)}+k(x)u_j^{(1,2)}=F_j^{(2)}m^{(1)}m^{(2)}
\ \ \hspace*{1.4cm} &\mathrm{in}\  \Omega,\medskip\\
\displaystyle -v\Delta m^{(1,2)}-\mathrm{div}(m^{(1)}\nabla u^{(2)})-\mathrm{div}(m^{(2)}\nabla u^{(1)})+r(x)m^{(1,2)}=0
\hspace*{1.50cm} &\mathrm{in}\ \Omega,\medskip\\
\displaystyle  u_j^{(1,2)}=m^{(1,2)}=0\hspace*{5.3cm} &\mathrm{on}\  \partial\Omega.
\end{cases}
\end{align}
Notice that $u^{(1)},u^{(2)}, m^{(1)}, m^{(2)}$ are not dependent on $F_j^{(2)}$. Subtracting  the first equation of \eqref{fum} with $j = 1,2,$ we find
\begin{align}
\label{F2}
\begin{cases}
\displaystyle-v\Delta p+k(x)p=(F_1^{(2)}-F_2^{(2)})m^{(1)}m^{(2)}
\ \ \hspace*{1.4cm} &\mathrm{in}\  \Omega,\medskip\\
\displaystyle  p=0\hspace*{5.3cm} &\mathrm{on}\  \partial\Omega.
\end{cases}
\end{align}
where $p=u_1^{(1,2)}-u_2^{(1,2)}$. Let $\varpi$ be a solution to the equation $-v\Delta \varpi+k\varpi=0$ in $\Omega$. With the same DN map, by differentiating $\epsilon_1$ and $\epsilon_2$ yields $\partial_\nu u_1^{(1,2)}=\partial_\nu u_2^{(1,2)}$ from \eqref{cod}. Then, multiplying both sides of \eqref{F2} by $\varpi$, an integration by parts, we derive that
\begin{align*}
 \int_{\Omega}(F_2^{(2)}-F_1^{(2)})m^{(1)}m^{(2)}\varpi dx=0,
\end{align*}
we take
\begin{align*}
m^{(2)}:&=e^{-\frac{1}{h}(\mathrm{i}\psi+\varphi)}(1+b_1),\\
\varpi:&=e^{\frac{1}{h}(\mathrm{i}\psi+\varphi)}(e^{\mathrm{i}\xi\cdot x}+b_2),
\end{align*}
as in \eqref{CGO}, where $\| b_j\|_{L^2(\Omega)}\leq Ch$ for $j=1,2$, $\xi\in \mathbb{R}^n$ and $\{ \lambda,\eta,\xi\}$ is an orthogonal triplet. Now, by Lemma \ref{fdata}, there exist a sequence of solution $\hat{m}_l\in C^{2+\alpha}(\overline{\Omega})$ to \eqref{um1j} and $\hat{\varpi}_l\in C^{2+\alpha}(\overline{\Omega})$ to the equation $-v\Delta \varpi+k\varpi=0$, such that
\begin{align*}
\lim_{l\rightarrow\infty}\|\hat{m}_l-m^{(2)}\|_{L^2(\Omega)}=0,\
\lim_{l\rightarrow\infty}\|\hat{\varpi}_l-\varpi\|_{L^2(\Omega)}=0.
\end{align*}
Let $l\rightarrow\infty$, we have
\begin{align*}
 \int_{\Omega}(F_2^{(2)}-F_1^{(2)})m^{(1)}(e^{\mathrm{i}\xi\cdot x}+b_1)(1+b_2)dx=0,
\end{align*}
as $h\rightarrow 0$, and combing with the Fourier transform of $\tilde{F}=(F_2^{(2)}-F_1^{(2)})m^{(1)}$ in $\Omega$ and vanish outside $\Omega$ is zero for every $\xi\in \mathbb{R}^n$, we have
\[
(F_2^{(2)}-F_1^{(2)})m^{(1)}=0\ \ \mathrm{in}\ \Omega.\]
Since $g_1>0$, according to the  maximum principle, we have $m^{(1)}>0$ in $\Omega$, thus $F_2^{(2)}=F_1^{(2)}$ in $\Omega$.
Finally, we can show that this is valid for $N\geq 3$ of \eqref{fg} by mathematical induction. In other words, for any $i\in N$, we have $F_1^{(i)}=F_2^{(i)}$. Hence, $F_1(x,m)=F_2(x,m)$ in $\Omega\times \mathbb{R}^+$.

\subsection{Proof of Theorem \ref{pdata}}\label{proof1.2}
The argument is similar to the proof of theorem \ref{alldata}, and we prove this result with partial data.

$\mathbf{Step\ I}$. Let us introduce the following boundary value for the system \eqref{um0}
\begin{align*}
f(x)=\sum_{l=1}^N\epsilon_lf_l\ \  \mathrm{and} \ g(x)=\sum_{l=1}^N\epsilon_lg_l \ \ \mathrm{for}\ x\in\partial\Omega,
\end{align*}
where $f_l\in\xi_2$ for $l=1,2...,N$, $g_l\in \xi_2^+$ for $l\neq 2$, and $\epsilon=(\epsilon_1,\epsilon_2,...,\epsilon_n)\in (\mathbb{R}^{N})^+$ with $|\epsilon|=\sum_{l=1}^N|\epsilon_l|$ small enough. It is important to note that $g_2$ can be chosen to be an arbitrary function in $C_0^{2+\alpha}(\mathcal{U}_+)$, and that $\epsilon_2g_2$ can be made sufficiently small such that $g(x)>0$.

\medskip

$\mathbf{Step\ II}$. The uniqueness of $k(x)$ and $r(x)$ is first proved by first-order linearization. For $j=1,2$,  $u_j^{(2)}$ and $m_j^{(2)}$ satisfy the following equations
\begin{align}
\label{um1jp}
\begin{cases}
\displaystyle-v\Delta u_j^{(2)}+k_ju_j^{(2)}=0
\ \ \hspace*{1.4cm} &\mathrm{in}\  \Omega,\medskip\\
\displaystyle -v\Delta m_j^{(2)}+r_jm_j^{(2)}=0
\hspace*{1.50cm} &\mathrm{in}\ \Omega,\medskip\\
\displaystyle  u_j^{(2)}=f_2,\  m_j^{(2)}=g_2\hspace*{5.3cm} &\mathrm{on}\  \partial\Omega.
\end{cases}
\end{align}
Since the procedure for proving $k(x)$ is similar to that for $r(x)$, we next give the procedure for proving only $r(x)$. Subtracting the second equation of \eqref{um1jp} for $j=1,2$, we have
\begin{align}
\label{w}
\begin{cases}
\displaystyle-v\Delta w+r_1w=(r_2-r_1)m_2^{(2)}
\ \ \hspace*{1.4cm} &\mathrm{in}\  \Omega,\medskip\\
\displaystyle  w=0\hspace*{5.3cm} &\mathrm{on}\  \partial\Omega,
\end{cases}
\end{align}
where $w=m_1^{(2)}-m_2^{(2)}$. There is  a solution
$$m^{(2)}_2=e^{\frac{1}{h}(\mathrm{i}\psi+\varphi)}(e^{\mathrm{i}x\cdot\xi}+b_2)\in H^1(\Omega)$$
to the second equation of \eqref{um1jp} with respect to $r_2(x)$ such that
$$\lim_{h\rightarrow 0}\|b_2\|_{L^2(\Omega)}=0,$$
where $\psi=\lambda\cdot x$, $\varphi=\eta\cdot x$ with $\lambda, \eta\in\mathbb{S}^{n-1}$, and $\{\lambda, \eta, \xi\}$ is an orthogonal triplet. Consider  the following equation
\begin{align}
\label{v11}
\begin{cases}
\displaystyle-v\Delta v_1+r_1v_1=0
\ \ \hspace*{1.4cm} &\mathrm{in}\  \Omega,\medskip\\
\displaystyle  v_1=0\hspace*{3.3cm} &\mathrm{on}\  \partial\Omega_{+,\varepsilon_0}(\lambda).
\end{cases}
\end{align}
From the Lemma \ref{part-sol}, there is a solution
$$v_1=e^{-\frac{1}{h}(i\psi+\varphi)}(1+b_1),$$
satisfies the equation \eqref{v11}, where $\lim_{h\rightarrow 0}\|b_1\|_{L^2(\Omega)}=0.$ Then by the Lemma \ref{prung}, there are two sequences of functions $\{\hat{m}_l\}_{l=1}^\infty$, $\{v_1^l\}_{l=1}^\infty\in C^{2+\alpha}(\overline{\Omega})$, such that $v_1^l$ are solutions to \eqref{v11}, and $\hat{m}_l$ are solutions to equation as follows
\begin{align}
\label{uk}
\begin{cases}
\displaystyle-v\Delta \hat{m}_l+r_1\hat{m}_l=0
\ \ \hspace*{1.4cm} &\mathrm{in}\  \Omega,\medskip\\
\displaystyle  \hat{m}_l=0\hspace*{3.3cm} &\mathrm{on}\  \partial\Omega_{-,\varepsilon_0}(\lambda),
\end{cases}
\end{align}
and $\hat{m}_l\rightarrow m_2^{(2)}$, $v_1^l\rightarrow v_1$ in $L^2(\tilde{\Omega})$ as $l\rightarrow\infty$.
For $j=1,2$, we define
$$S_j=\{z\in H^1(\Omega)|(-v\Delta+r_j)z=0\ \ \mathrm{in}\ \Omega\},$$
and the map $\mathcal{M}:S_1\rightarrow S_2$ is defined by
$$\mathcal{M}(z_1)=z_2,$$
where $z_2$ is the solution to
\begin{align}
\label{w}
\begin{cases}
\displaystyle(-v\Delta +r_2)z_2=0
\ \ \hspace*{1.4cm} &\mathrm{in}\  \Omega,\medskip\\
\displaystyle  z_2=z_1\hspace*{5.3cm} &\mathrm{on}\  \partial\Omega.
\end{cases}
\end{align}
By the trace Theorem, $z_1|_{\partial\Omega}\in H^{1/2}(\partial\Omega)$ and the map $\mathcal{M}$ is well define. Hence, we have
\begin{align}
\label{w}
\begin{cases}
\displaystyle(-v\Delta +r_2)(\mathcal{M}(\hat{m}_l)-\hat{m}_l)=(r_1-r_2)\hat{m}_l
\ \ \hspace*{1.4cm} &\mathrm{in}\  \Omega,\medskip\\
\displaystyle \mathcal{M}(\hat{m}_l)-\hat{m}_l=0\hspace*{5.3cm} &\mathrm{on}\  \partial\Omega.
\end{cases}
\end{align}
Multiplying both sides of equation \eqref{w} by the function $v_1^l$
and integration by parts implies
\begin{align}\label{kk}
\int_{\Omega}(r_1-r_2)\hat{m}_lv_1^ldx=\int_{\partial \Omega}v\partial_{\nu}(\mathcal{M}(\hat{m}_l)-\hat{m}_l)v_1^ldS.
\end{align}
 We now claim that the equation \eqref{kk} is zero holds for $\lambda\in \{ \lambda\in \mathbb{S}^{n-1}||\lambda-\lambda'|<\varepsilon_0\}$. We know that $\mathcal{U}_{\pm}$ is a neighborhood of $\partial\Omega_{\pm}(\lambda')$, then
 $$\{x\in \zz|0<\lambda'\cdot\nu<2\varepsilon_0\}\subset\mathcal{U}_{-},$$
and
$$\{x\in \zz|\lambda'\cdot\nu>-2\varepsilon_0\}\subset\mathcal{U}_{+}.$$
Therefore, we have
\begin{align*}
\mathrm{supp}\ \hat{m}_l|_{\partial\Omega}\subset \{x\in \partial \Omega|\lambda\cdot\nu\geq -\varepsilon_0\}\subset\{x\in \partial \Omega|\lambda'\cdot\nu>-2 \varepsilon_0\}\subset \mathcal{U}_{+},
\end{align*}
and
$$\{ x\in\partial\Omega| \lambda'\cdot\nu\geq 2\varepsilon_0\}\subset \partial\Omega_{+,\varepsilon_0}(\lambda).$$
Note that $\hat{m}_l|_{\partial\Omega}=\mathcal{M}({\hat{m}}_l)|_{\partial\Omega}\in C^{2+\alpha}_0(\mathcal{U}_+)$ and $v_1^l=0$ on $\partial\Omega_{+,\varepsilon_0}(\lambda)$.
Then we derive that
\begin{align*}
\small
\bigg{|}\int_{\partial \Omega}v\partial_{\nu}(\mathcal{M}(\hat{m}_l)-\hat{m}_l)v_1^ldS\bigg{|}&=
\bigg{|}\int_{\lambda'\cdot\nu\geq 2\varepsilon_0}v\partial_{\nu}(\mathcal{M}(\hat{m}_l)-\hat{m}_l)v_1^ldS\\&+
\int_{0<\lambda'\cdot\nu< 2\varepsilon_0}v\partial_{\nu}(\mathcal{M}(\hat{m}_l)-\hat{m}_l)v_1^ldS\\&+
\int_{\lambda'\cdot\nu\leq 0}v\partial_{\nu}(\mathcal{M}(\hat{m}_l)-\hat{m}_l)v_1^ldS\bigg{|}\\
&=0,
\end{align*}
thus the claim is holds. From $r_1=r_2$ in $\Omega'$, we have
$$\int_{\tilde{\Omega}}(r_1-r_2)\hat{m}_lv_1^ldx=0,$$
as $l\rightarrow\infty$ and $h\rightarrow 0$, we have
\begin{align}\label{pep}
 &\int_{\tilde{\Omega}}(r_1-r_2)e^{\mathrm{i}\xi\cdot x}dx=0
\end{align}
 holds for
$\lambda\in \{ \lambda\in \mathbb{S}^{n-1}||\lambda'-\lambda|<\varepsilon_0\}$. It follows that \eqref{pep} holds for $\xi$ in an open cone in $\mathbb{R}^n$. The Fourier transform of $r=r_1-r_2$, which is compactly supported in $\tilde{\Omega}$ and vanishes outside $\tilde{\Omega}$, is zero in an open set. By the Paley-Wiener theorem, the Fourier transform is analytic, and this implies that $r\equiv 0$. Therefore, we have  $r_1=r_2$ in ${\Omega}$.

\medskip

$\mathbf{Step\ III}$. We consider a second order linearization. For $j=1,2$, we have $u_j^{(1,2)}$ and  $m_j^{(1,2)}$ satisfy the following equations
\begin{align}
\label{umN12}
\begin{cases}
\displaystyle-v\Delta u_j^{(1,2)}+\nabla u^{(1)}\cdot \nabla u^{(2)}+k(x)u_j^{(1,2)}=F_j^{(2)}m^{(1)}m^{(2)}
\ \ \hspace*{1.4cm} &\mathrm{in}\  \Omega,\medskip\\
\displaystyle -v\Delta m_j^{(1,2)}-\mathrm{div}(m^{(1)}\nabla u^{(2)})-\mathrm{div}(m^{(2)}\nabla u^{(1)})+r(x)m_j^{(1,2)}=0
\hspace*{1.50cm} &\mathrm{in}\ \Omega,\medskip\\
\displaystyle  u_j^{(1,2)}=m_j^{(1,2)}=0 \hspace*{5.3cm} &\mathrm{on}\  \partial\Omega.
\end{cases}
\end{align}
Subtracting the first equation of \eqref{umN12} with $j = 1,2,$ it follows that
\begin{align*}
\begin{cases}
\displaystyle-v\Delta p+k(x)p=(F_1^{(2)}-F_2^{(2)})m^{(1)}m^{(2)}
\ \ \hspace*{1.4cm} &\mathrm{in}\  \Omega,\medskip\\
\displaystyle  p=0\hspace*{5.3cm} &\mathrm{on}\  \partial\Omega,
\end{cases}
\end{align*}
where $p=u_1^{(1,2)}-u_2^{(1,2)}$. Let $\vartheta$ is a solution to the equation \begin{align}
\label{vv1}
\begin{cases}
\displaystyle-v\Delta \vartheta+k\vartheta=0
\ \ \hspace*{1.4cm} &\mathrm{in}\  \Omega,\medskip\\
\displaystyle  \vartheta=0\hspace*{3.3cm} &\mathrm{on}\  \partial\Omega_{+,\varepsilon_0}(\lambda).
\end{cases}
\end{align}
 From the Lemma \ref{prung}, there are two sequences of functions $\{\hat{m}_l\}_{l=1}^\infty$, $\{\vartheta_l\}_{l=1}^\infty \in C^{2+\alpha}(\overline{\Omega})$, such that $\vartheta_l$ are solutions to \eqref{vv1}, and $\hat{m}_l$ are solutions to equation in \eqref{uk}, and $\hat{m}_l\rightarrow m^{(2)}, \vartheta_l\rightarrow \vartheta$ in $\tilde{\Omega}$. Similar to the proof in step II, we obtain that $$\int_{{\Omega}}(F^{(2)}_1-F^{(2)}_2)m^{(1)}\hat{m}_l\vartheta_ldx=0$$  also holds for $\lambda$ sufficiently close to $\lambda'$ on the unit sphere. As $l\rightarrow \infty,h\rightarrow 0$, and $F_1^{(2)}=F_2^{(2)}$ in $\Omega'$, we have
\begin{align*}
\int_{\tilde{\Omega}}(F_2^{(2)}-F_1^{(2)})m^{(1)}e^{\mathrm{i}x\cdot\xi}dx=0.
\end{align*}
 Since the Fourier transform of $F=(F_2^{(2)}-F_1^{(2)})m^{(1)}$ in $\tilde{\Omega}$ and vanish outside $\tilde{\Omega}$ is zero in an open set. Thus we have
\[
(F_2^{(2)}-F_1^{(2)})m^{(1)}=0 \ \ \mathrm{for}\ x\in\tilde{\Omega}.\]
Due to $g_1>0$, by the  maximum principle we have $m^{(1)}>0$ in $\Omega$, then we obtain $F_2^{(2)}=F_1^{(2)}$ in $\tilde{\Omega}$. According to $F_2^{(2)}=F_1^{(2)}$ in ${\Omega'}$, we derive that $F_2^{(2)}=F_1^{(2)}$ in ${\Omega}$. Finally, by the mathematical induction, we can show that this is true for $N\geq 3$. In other words, for any $i\in N$, we have $F_1^{(i)}=F_2^{(i)}$.
Hence, $F_1(x,z)=F_2(x,z)$ in $\Omega\times \mathbb{R}^+$.
\section*{Acknowledgments}
\addcontentsline{toc}{section}{Acknowledgments}
 The work of H. Liu is supported by the Hong Kong RGC General Research Funds (projects 11311122, 11300821 and 12301420),  the NSFC/RGC Joint Research Fund (project N\_CityU101/21), and the ANR/RGC Joint Research Grant, A\_CityU203/19. The work of G. Zheng were supported by the NSF of China (12271151) and NSF of Hunan (2020JJ4166).

\end{document}